\documentclass[12pt,twoside,reqno]{amsart}
\usepackage[utf8]{inputenc}
\usepackage[T1]{fontenc}
\usepackage{lmodern}
\usepackage[a4paper]{geometry}
\usepackage{xcolor}
\usepackage{graphicx}
\usepackage{amsmath}
\usepackage{amsfonts}
\usepackage{amssymb}
\usepackage{mathtools}
\usepackage{mathrsfs}
\usepackage{cite}
\usepackage[english]{babel}
\newcommand{\mA}{\mathcal{A}}
\newcommand{\rd}{\mathrm{d}}
\newcommand{\mL}{\mathcal{L}}
\newcommand{\mK}{\mathcal{K}}

\newtheorem{theorem}{Theorem}[section]
\newtheorem{corollary}[theorem]{Corollary}
\newtheorem{lemma}[theorem]{Lemma}
\newtheorem{proposition}[theorem]{Proposition}

\newtheorem{remark}[theorem]{Remark}
\numberwithin{equation}{section}
\begin{document}
\title{Global Boundedness Induced by Asymptotically Non-degenerate Motility in a Fully Parabolic Chemotaxis Model with Local Sensing} 
\thanks{\texttt{orcid. JJ: 0000-0003-1452-1651, PhL: 0000-0003-3091-8085}}
\author{Jie Jiang}
\address{Innovation Academy for Precision Measurement Science and Technology, Chinese Academy of Sciences, Wuhan 430071, HuBei Province, P.R. China}
\email{jiang@apm.ac.cn, jiang@wipm.ac.cn}
	
\author{Philippe Lauren\c{c}ot}
\address{Laboratoire de Math\'ematiques (LAMA) UMR~5127, Universit\'e Savoie Mont Blanc, CNRS\\
F--73000 Chamb\'ery, France}
\email{philippe.laurencot@univ-smb.fr}

	
\date{\today}
	
\begin{abstract}
A fully parabolic chemotaxis model  of Keller--Segel type with local sensing is considered. The system features a signal-dependent asymptotically non-degenerate motility function,  which accounts for a repulsion-dominated chemotaxis. Global boundedness of classical solutions is proved for an initial Neumann boundary value problem of the system in any space dimension. In addition, stabilization towards the homogeneous steady state is shown, provided that the motility is monotone non-decreasing. The key steps of the proof consist of the introduction of several auxiliary functions and a refined comparison argument, along with uniform-in-time estimates for a functional involving nonlinear coupling between the unknowns and auxiliary functions.
\end{abstract}
	
	\maketitle
	%
	%
\pagestyle{myheadings}
\markboth{\sc{J. Jiang \& Ph. Lauren\c{c}ot}}{\sc{Boundedness   in a Keller--Segel Model with Asymptotically Non-degenerate Motility}}
	
\section{Introduction}
	
Let $\Omega$ be a smooth bounded domain of $\mathbb{R}^N$, $N\ge 1$, and consider the following initial-boundary value problem
\begin{subequations}\label{ks}
	\begin{align}
		& \partial_t u = \Delta\big( u \gamma(v)\big), \qquad &(t,x) &\in  (0,T_{\text{max}}) \times\Omega, \label{ks1}\\			& \tau \partial_t v = \Delta v -v + u ,  \qquad &(t,x) &\in (0,T_{\text{max}}) \times\Omega, \label{ks2}\\			& \nabla\big( u\gamma(v)\big)\cdot \mathbf{n} = \nabla v\cdot \mathbf{n} = 0, \qquad &(t,x) &\in (0,T_{\text{max}}) \times\partial\Omega, \label{ks3}\\			& (u,v)(0) = (u^{in},v^{in}), \qquad &x &\in\Omega, \label{ks4}
	\end{align}
\end{subequations} 
where $\mathbf{n}$ denotes the outward unit normal vector field to $\partial\Omega$.
	 
System~\eqref{ks} was originally proposed by Keller and Segel in their seminal work \cite{1971KS} to model the chemotaxis phenomenon induced by a local sensing mechanism. In some  recent bio-physics work \cite{PhysRevLett.108.198102},  its non-homogeneous version with an additional logistic source in the first equation is applied to model formation of stripe patterns. In~\eqref{ks}, $u$ and $v$ denote the density of cells and the concentration of chemical signal, respectively. The motility $\gamma(v)$ explicitly depends on $v$ with a motility function $\gamma$, and its monotonicity accounts for the influence of chemical stimuli on cellular movement. Indeed, through  a direct expansion  $\Delta(u\gamma(v))=\mathrm{div}(\gamma(v)\nabla u+u\gamma'(v)\nabla v)$, it is easily seen that cells are attracted when $\gamma'<0$, while they are repelled when $\gamma'>0$.
	
Recent progress on the theoretical analysis on~\eqref{ks} has uncovered certain interesting relation between its dynamics and the properties of $\gamma$. Roughly speaking, in the monotone non-increasing case, a classical solution always exists globally in any space dimension $N\geq1$, whatever the value of $\tau\ge 0$ \cite{JiLa2021,FuSe2022a}, while its boundedness is closely related to $N$ and the decay rate of $\gamma$ at infinity \cite{JiLa2021,JLZ2022,FuSe2022b,FuJi2021b}. For the specific choice of motility function $\gamma(s)=e^{-s}$, unbounded solutions are constructed in any space dimension $N\geq2$ \cite{FuJi2020,FuJi2021a,JW2020,FuSe2022a}, which well captures a concentration trend of the population due to the chemoattraction effect.

 On the contrary, a monotone non-decreasing motility seems to stabilize the dynamics. In \cite{JL2024}, we analyze the dynamics of~\eqref{ks} with $\tau=0$ and $\gamma$ being asymptotically unbounded, corresponding to a chemorepulsion dominated situation. It is proved in that case that classical solutions exist globally and are  always uniformly-in-time bounded. Moreover, if $\gamma$ is monotone non-decreasing, then explicit upper bounds for both $u$ and $v$ are derived, and their stabilization towards a homogeneous steady state is shown as well. In contrast, when $\tau>0$, neither global existence, nor boundedness of classical solutions, have been thoroughly investigated and the situation is yet unclear. As far as we know, the only contribution in that direction deals with bounded positive motility functions, for which the existence and uniqueness of a uniformly-in-time bounded global classical solution are established in arbitrary space dimension in \cite{XiaoJiang2022}.
 
 In the present contribution, it is our aim to verify the same boundedness-enforcing effect of an asymptotically non-degenerate motility function in \eqref{ks} when $\tau>0$, which in particular allows for unbounded motility function. The current work, together with our previous one \cite{JL2024}, then implies that, when chemorepulsion dominates in the Keller--Segel model with local sensing, classical solutions always exist globally and stay bounded.
	
Throughout this paper, we assume that
\begin{equation}
	\gamma\in C^3((0,\infty)), \qquad \gamma>0 \;\;\text{ on }\; (0,\infty), \label{hypgam1}
\end{equation}
and moreover, that $\gamma$ is asymptotically non-degenerate, or in other words,  it is asymptotically bounded from below in the sense that
\begin{equation}\label{ginf}
	\gamma_\infty:=\liminf\limits_{s\rightarrow\infty}\gamma(s)>1/\tau.
\end{equation}
For the initial conditions, we assume that
\begin{equation}\label{ini}
	\begin{split}
		& \left( u^{in}, v^{in} \right)\in W^{1,N+1}(\Omega;\mathbb{R}^2)\,, \quad u^{in}\not\equiv 0\,, \\
		& u^{in}\ge 0\,, \quad  v^{in}> 0\quad  \mbox{in } \bar\Omega\,.
	\end{split}
\end{equation} 	

Our first main result states that the initial-boundary value problem \eqref{ks} has a unique global bounded classical solution provided $\gamma$ is asymptotically bounded from below by $1/\tau$.

\begin{theorem}\label{TH1}
Assume that $\gamma$ satisfies~\eqref{hypgam1} and~\eqref{ginf}. For any given  initial conditions $(u^{in},v^{in})$ satisfying~\eqref{ini}, the initial-boundary value problem~\eqref{ks} has  a unique global non-negative classical solution $(u,v)\in C([0,\infty)\times \bar{\Omega};\mathbb{R}^2)\cap C^{1,2}((0,\infty)\times \bar{\Omega};\mathbb{R}^2)$, which satisfies the conservation of mass
\begin{equation}
	\|u(t)\|_1 = m|\Omega| \triangleq \|u^{in}\|_1, \qquad t\ge 0, \label{cm}
\end{equation}
and is uniformly-in-time bounded in the sense that
\begin{equation*}
	\sup\limits_{t\geq0}\{\|u(t)\|_\infty+\|v(t)\|_\infty\}<\infty. 
\end{equation*}
\end{theorem}	

In the particular case where the motility function $\gamma$ is non-decreasing, which corresponds to a fully chemorepulsive behavior, we may relax the asymptotic lower bound~\eqref{ginf} and prove that bounded classical solutions exist globally and converge to the homogeneous steady state as time increases to infinity.

\begin{theorem}\label{TH2}
Assume that \eqref{hypgam1} and that $\gamma'\geq 0$. For any given  initial conditions $(u^{in},v^{in})$ satisfies~\eqref{ini} with $m=\|u^{in}\|_1/|\Omega|$, the initial-boundary value problem~\eqref{ks} has  a unique global non-negative classical solution $(u,v)\in C([0,\infty)\times \bar{\Omega};\mathbb{R}^2)\cap C^{1,2}((0,\infty)\times \bar{\Omega};\mathbb{R}^2)$, which is uniformly-in-time bounded and moreover, 
\begin{equation*}
	\lim\limits_{ t\rightarrow \infty}\left(\|u(t)-m\|_\infty+\|v(t)-m\|_{W^{1,\infty}}\right)=0.
\end{equation*}		
\end{theorem}

\begin{remark}
	Let us emphasize that Theorem~\ref{TH2} provides the global existence of classical solutions to the chemorepulsion model with local sensing in arbitrary space dimension, a feature with contrasts markedly with the current knowledge on the classical Keller-Segel chemorepulsion model, for which such a result is only available in space dimensions $N\in \{1,2\}$ \cite{CLMR}. 
\end{remark}

\bigskip
In view of the signal-dependent feature of $\gamma$,  a positive lower bound, along with an upper one, for the signal concentration plays a crucial role in the analysis. While the former can be obtained by a well-established result due to \cite[Lemma~2.6]{Fuji2016} and the conservation of mass~\eqref{cm}, the main obstacle lies in deriving the latter upper bound estimate for $v$. In a series of recent research \cite{FuJi2020,FuJi2021a,XiaoJiang2022}, a comparison argument is developed to achieve this goal. However, the essential assumption $\limsup\limits_{s\rightarrow\infty}\gamma(s)<1/\tau$ is needed therein, which obviously excludes unbounded motility functions.  In our recent work \cite{JL2024}, the difficulty stemming from the  unboundedness of motility functions is overcome by the development of new comparison techniques, taking advantage of the delicate duality structure of the system~\eqref{ks} when $\tau=0$, but the technique developed in \cite{JL2024} does not seem to extend to positive values of $\tau$. 

In the present work, we achieve a fundamental refinement of the key idea in \cite{FuJi2021a}, allowing us to cope with  unbounded motility function in~\eqref{ks}. Specifically, besides the introduction of the auxiliary function $w$ satisfying an elliptic problem as done in \cite{FuJi2021a}, we construct a new auxiliary function $\Psi$, which is the solution of a heat equation with the external source term $u\gamma(v)$. Making use of the nice duality structure, we are able to establish an important identity, unveiling the intrinsic relation between the original unknowns $(u,v)$ and the auxiliary functions $(w,\Psi)$, see Lemma~\ref{lmkid2}. A suitable application of parabolic comparison principles then shows a pointwise upper control of $v$ by $\Psi$, see Lemma~\ref{lem3a} and Lemma~\ref{lem3}. Next, an iteration argument based on elliptic regularity, along with the above mentioned identity, surprisingly gives rise to an upper control of $v$ by an $L^1$-estimate on $\Psi$, see Lemma~\ref{cor1}. The final step is the derivation of a uniform-in-time $L^1$-estimate for $\Psi$ by constructing energy estimates for a functional involving a nonlinear coupling term between the unknowns and auxiliary functions, using again the above mentioned identity.  

\medskip
The paper is organized as follows. In section~\ref{sec.2}, the local well-posedness of~\eqref{ks} is established in the framework of classical solutions, along with some useful lemmas. In section~\ref{sec.3}, we introduce the auxiliary functions and establish two key identities that play important roles in the proofs. In section~\ref{sec.4}, we derive a pointwise control on $v$ by $L^1$-bounds for the new auxiliary function via comparison and iteration. We first consider the simpler case $\gamma'\geq0$, and then generalize the result to non-monotone motility functions. In  section~\ref{sec.5}, we establish uniform-in-time $L^1$-estimates of the auxiliary function  by energy methods. We prove Theorem~\ref{TH1} and Theorem~\ref{TH2} in the last section.

\section{Preliminaries}\label{sec.2}

In this section, we recall some useful results. We begin with the local well-posedness in a suitable functional setting, which mainly follows from the theory developed by Amann in \cite{Aman1988, Aman1989, Aman1990, Aman1993} and the comparison principle, along with positivity properties of the heat equation. A proof can be found in \cite[Proposition~2.1]{JLZ2022}.

\begin{proposition}\label{local}
	Suppose that $\gamma$ and $(u^{in},v^{in})$ satisfy~\eqref{hypgam1} and~\eqref{ini}, respectively. Then there exists $T_{\mathrm{max}} \in (0, \infty]$ such that problem~\eqref{ks} has a unique non-negative classical solution 
    \begin{equation*}
        (u,v)\in C([0,T_{\mathrm{max}})\times \bar{\Omega};\mathbb{R}^2)\cap C^{1,2}((0,T_{\mathrm{max}})\times \bar{\Omega};\mathbb{R}^2),
    \end{equation*}
    which satisfies the mass conservation
	\begin{equation}
		\int_\Omega u(t,x)\ \mathrm{d} x = m |\Omega| = \int_\Omega u^{in}(x)\ \mathrm{d} x >0
		\quad \text{for\ all}\ t \in (0,T_{\mathrm{max}})\,.\label{e0}
	\end{equation}	
	Moreover,  there holds
	\begin{equation}
		\|v(t)\|_{1}=\|v^{in}\|_1e^{-t/\tau}+\|u^{in}\|_1 (1-e^{-t/\tau})\leq\max\{\|u^{in}\|_1,\|v^{in}\|_1\}\quad \text{for\ all}\ t \in (0,T_{\mathrm{max}}),\label{nup}
	\end{equation}
	and there is $v_*>0$ depending only on $\Omega$, $v^{in}$, and $\|u^{in}\|_1$ such that
	\begin{equation}
		v(t,x) \ge v_*\,, \qquad (t,x)\in [0,T_{\mathrm{max}})\times \bar{\Omega}\,. \label{e00}
	\end{equation}
	Finally, if $T_{\mathrm{max}}<\infty$, then
	\begin{equation*}
		\limsup\limits_{t\nearrow T_{\mathrm{max}}}\|u(t)\|_{\infty}=\infty.
	\end{equation*}
\end{proposition}

We next show that an $L^\infty$-estimate for $v$ on  $(0,T)$ for some $T>0$ guarantees that $T_{\mathrm{max}}\ge T$. 

\begin{proposition}\label{prop2.2}
	Under the assumption of Proposition~\ref{local}, if there is $T>0$ such that
	\begin{equation}
		\mathcal{V}(T) \triangleq \sup_{[0,T]\cap [0,T_{\mathrm{max}})}\big\{ \|v(t)\|_\infty \big\} <\infty\,, \label{vvv}
	\end{equation} 
	then $T_{\mathrm{max}}\ge T$ and 
	\begin{equation*}
		\mathcal{U}(T) \triangleq \sup_{[0,T]}\big\{ \|u(t)\|_\infty \big\} < \infty\,. 
	\end{equation*}
	In addition, if~\eqref{vvv} holds true for all $T>0$ and there is $\mathcal{V}_\infty>0$ such that $\mathcal{V}(T)\le\mathcal{V}_\infty$ for all $T>0$, then there is $\mathcal{U}_\infty>0$ such that  $\mathcal{U}(T)\le \mathcal{U}_\infty$ for all $T>0$.
\end{proposition}

\begin{proof}
	The proof relies on a well-established argument already described in \cite{JLZ2022}, see also \cite{FuSe2022a,FuSe2022b}, and we refer the interested reader to \cite[Section~1]{XiaoJiang2022} for an outline of the key steps of the proof.  As mentioned in \cite{XiaoJiang2022}, no monotonicity property of $\gamma$ is needed, provided there are positive upper and lower bounds on $\gamma(v)$.
\end{proof}

Next, we let $(\cdot)_+=\max\{\cdot,0\}$ and recall the following result,  see \cite[(9.2)~Proposition]{Aman1983}, \cite[Lemme~3.17]{BeBo1999}, or \cite[Lemma~2.2]{AhnYoon2019}.

\begin{lemma}\label{lm2}
	Let $f\in  L^1(\Omega)$. For any $1\leq q< \frac{N}{(N-2)_+}$, there exists a positive constant $C(q)$ depending only on $\Omega$ and $q$ such that the solution $z \in W^{1,1}(\Omega)$ to
	\begin{equation*}
		-\Delta z+ z=f,\qquad x\in\Omega\,,\qquad\qquad
		\nabla z\cdot \mathbf{n}=0\,,\qquad x\in\partial\Omega\,,
	\end{equation*}
	satisfies $\|z\|_q \leq C(q) \|f\|_1$.
\end{lemma}

\section{Auxiliary functions}\label{sec.3}

In this section, we aim to improve the approach  proposed in \cite{FuJi2021a} and later developed in \cite{JLZ2022, FuSe2022b,XiaoJiang2022}. The main novelty is to introduce several auxiliary functions in connection with the original unknown functions $u$ and $v$ via linear parabolic/elliptic equations involving  $(u,v)$-dependent non-homogeneous terms.

\subsection{Auxiliary problems}\label{sec.3.1}

To begin with, we recall the definition of the first auxiliary function already introduced in \cite{FuJi2021a}. Let $\mathcal{A}$ be the Laplace operator on $L^2(\Omega)$ supplemented with homogeneous Neumann boundary conditions; that is, 
\begin{equation*}
	\mathrm{dom}(\mathcal{A}) \triangleq \{ z \in H^2(\Omega)\ :\ \nabla z\cdot \mathbf{n} = 0 \;\text{ on }\; \partial\Omega\}\,, \qquad  \mathcal{A}z \triangleq - \Delta z +  z\,, \quad z\in \mathrm{dom}(\mathcal{A})\,. 
\end{equation*} 
It is well-known that $\mathcal{A}$ generates an analytic semi-group $\big( e^{-t\mA} \big)_{t\ge 0}$ on $L^p(\Omega)$ and is invertible on $L^p(\Omega)$ for all $p\in (1,\infty)$. We then set 
\begin{equation*}
	w(t) \triangleq \mathcal{A}^{-1}[u(t)]\,,  \qquad t\in [0,T_{\mathrm{max}})\,.
\end{equation*}
In other words, $w$ satisfies the following Helmholtz problem:
\begin{equation*}
\begin{aligned}
	& \mathcal{A}[w]=-\Delta w+w=u, \qquad &(t,x) &\in  (0,T_{\text{max}}) \times\Omega, \\
	& \nabla w\cdot \mathbf{n} = 0, \qquad &(t,x) &\in  (0,T_{\text{max}}) \times\partial\Omega.
\end{aligned}
\end{equation*}
Note that $w(t,x)\geq 0$ on $ [0,T_{\mathrm{max}})\times\Omega$, due to the non-negativity of $u$ and the elliptic comparison principle. Thanks to the time continuity of $u$, 
\begin{equation*}
	w^{in}\triangleq w(0)=\mA^{-1}[u^{in}]\,,
\end{equation*} 
and it follows from the regularity assumption \eqref{ini} on the initial conditions that $w^{in}$ belongs to $W^{3,N+1}(\Omega)$, and in particular to $C^2(\bar{\Omega})$. 

Next, we define the linear parabolic operator $\mL:=\tau\partial_t+\mA$, and introduce the second auxiliary function $\Psi$ as the unique solution to
\begin{equation}\label{Psi}
\begin{aligned}
	& \mathcal{L}[\Psi]=\tau \partial_t\Psi-\Delta\Psi+\Psi =u\gamma(v), \qquad &(t,x) &\in  (0,T_{\text{max}}) \times\Omega, \\
	& \nabla \Psi\cdot \mathbf{n} = 0, \qquad &(t,x) &\in  (0,T_{\text{max}}) \times\partial\Omega, \\
	&\Psi(0) = 0, \qquad &x &\in \Omega.
\end{aligned}
\end{equation}
According to the parabolic comparison principle, we infer from the non-negativity of $u\gamma(v)$ that $\Psi(t,x)\geq 0$ for $(t,x)\in(0,T_{\text{max}}) \times\Omega$. For further use, we set $\psi\triangleq\mA^{-1}[\Psi]\geq0$ and readily deduce from~\eqref{Psi} that $\psi$ satisfies
\begin{equation}\label{psi}
\begin{aligned}
	& \mathcal{L}[\psi]=\tau \partial_t\psi-\Delta\psi+\psi =\mA^{-1}[u\gamma(v)], \qquad &(t,x) &\in  (0,T_{\text{max}}) \times\Omega, \\
	& \nabla \psi\cdot \mathbf{n} = 0, \qquad &(t,x) &\in  (0,T_{\text{max}}) \times\partial\Omega, \\
	&\psi(0) = 0, \qquad &x &\in \Omega.
\end{aligned}
\end{equation}
Lastly, set $\eta^{in}\triangleq w^{in}-v^{in}$ and let $\eta$ be the solution to the following heat equation:
\begin{equation}\label{eta}
\begin{aligned}
	& \mL[\eta]=\tau \partial_t\eta-\Delta\eta+\eta =0, \qquad &(t,x) &\in  (0,\infty) \times\Omega, \\
	& \nabla \eta\cdot \mathbf{n} = 0, \qquad &(t,x) &\in  (0,\infty) \times\partial\Omega, \\
	&\eta(0) = \eta^{in}, \qquad &x &\in \Omega;
\end{aligned}
\end{equation}
that is, $\eta(t,x)= e^{-t\mA/\tau}[\eta^{in}](x)$ for $(t,x)\in (0,\infty)\times\Omega$. Since both $u^{in}$ and $v^{in}$ are bounded due to~\eqref{ini} and the continuous embedding of $W^{1,N+1}(\Omega)$ in $L^\infty(\Omega)$ and since $\|\mA^{-1}[u^{in}]\|_\infty \le \|u^{in}\|_\infty$ by the elliptic comparison principle, the parabolic comparison principle applied to~\eqref{eta} readily implies the boundedness of $\eta$ with
\begin{equation}\label{etabd}
	\sup\limits_{ t\geq0}\|\eta(t)\|_\infty\leq\|\eta^{in}\|_\infty=\|\mA^{-1}[u^{in}]-v^{in}\|_\infty\leq \max\{\|u^{in}\|_\infty,\|v^{in}\|_\infty\}.
\end{equation}

\subsection{Two key identities} \label{sec.3.2}

In terms of the auxiliary functions, we find two useful key identities that play fundamental roles in the forthcoming analysis. The first one stated below is uncovered in \cite[Lemma~5]{FuJi2021a} (see also \cite[Lemma~3.1]{FuJi2020}) and can be obtained by taking $\mA^{-1}$ on both sides of~\eqref{ks1} and using the definition of $w$.

\begin{lemma}\label{lmkid1} For all $(t,x)\in[0,T_{\text{max}}) \times\Omega$, there holds
	\begin{equation}\label{keyid}
		\partial_t w + u\gamma(v) = \mA^{-1}[u\gamma(v)]\,.
	\end{equation}
\end{lemma}

Building upon~\eqref{keyid} and the definition of the auxiliary functions $(\Psi,\psi,\eta)$, we  are now ready to  derive the other key identity.

\begin{lemma}\label{lmkid2}For all $(t,x)\in[0,T_{\text{max}}) \times\Omega$, there holds
	\begin{equation}\label{kid2}
		w+\tau\Psi=v+\tau\psi+\eta\,.
	\end{equation}
\end{lemma}

\begin{proof}
	Multiplying the key identity~\eqref{keyid} by $\tau$ and adding $u=\mA[w]$ to both sides of the resulting identity, we obtain that
	\begin{equation*}
		\tau \partial_t w +\mA[w]+\tau u\gamma(v)=\tau\mA^{-1}[u\gamma(v)]+u.
	\end{equation*}
Since $u=\mL[v]$ by~\eqref{ks2}, it readily follows from the definitions~\eqref{Psi} and~\eqref{psi} of $\Psi$ and $\psi$ that the above identity can be rewritten as
\begin{equation*}
	\mL[w+\tau\Psi]=\mL[v+\tau\psi].
\end{equation*}	
Since $\mL[\eta]=0$ and $\eta^{in}=w^{in}-v^{in}$ by~\eqref{eta}, one easily verifies that
\begin{align*}
	&  \mL[w+\tau\Psi-v-\tau\psi-\eta]=0, \qquad &(t,x) &\in  (0,T_{\text{max}}) \times\Omega, \\
	& \nabla (w+\tau\Psi-v-\tau\psi-\eta)\cdot \mathbf{n} = 0, \qquad &(t,x) &\in  (0,T_{\text{max}}) \times\partial\Omega, \\
	&(w+\tau\Psi-v-\tau\psi-\eta)(0) = 0, \qquad &x &\in \Omega.
\end{align*}	
Then identity~\eqref{kid2} follows from the uniqueness of classical solutions to heat equations. This completes the proof.
\end{proof}

We conclude this section with a straightforward consequence of~\eqref{kid2} and the boundedness~\eqref{etabd} of $\eta$.

\begin{corollary}\label{corA}
For all $(t,x)\in [0,T_{\text{max}}) \times\bar{\Omega}$,
\begin{subequations}
\begin{equation}\label{corA1}
	v+\tau\psi\leq w+\tau\Psi+\max\{\|u^{in}\|_\infty,\|v^{in}\|_\infty\},
\end{equation}and
\begin{equation}\label{corA2}
 w+\tau\Psi\leq v+\tau\psi+\max\{\|u^{in}\|_\infty,\|v^{in}\|_\infty\}.
\end{equation}
\end{subequations}
\end{corollary}

\section{Intertwined estimates on $(v,w,\Psi,\psi)$ via comparison arguments}\label{sec.4}

Complementing additional assumptions on $\gamma$, we derive a direct upper control of $v$ by the auxiliary function $\Psi$ via comparison techniques. 

\subsection{Monotone non-decreasing motility functions}\label{sec.4.1}

We begin with the simple case $\gamma'\geq0$. We first remark that, if $\gamma$ has a generic upper bound on $(0,\infty)$, then the monotonicity of $\gamma$ and~\eqref{e00} entail that $\gamma(v)$ is also bounded from below by a positive constant $\gamma(v_*)$ on $[0,T_{\mathrm{max}})\times\Omega$. We may thus directly apply \cite[Theorem~1.2]{XiaoJiang2022} to get global existence and uniform-in-time boundedness of $v$ and $u$. We are then left with the case where $\gamma$ is non-decreasing on $(0,\infty)$ and is unbounded, i.e., 
 \begin{equation}\label{ginfity0}
 	\lim\limits_{s\rightarrow\infty}\gamma(s)=\infty.
 \end{equation} 
  
\begin{lemma}\label{lem3a}
Assume that $\gamma$ satisfies~\eqref{hypgam1}, $\gamma'\geq0$ on $(0,\infty)$ and~\eqref{ginfity0}. For any $\varepsilon>0$, there is $C_\varepsilon>0$ depending on $\varepsilon$, $\gamma$ and the initial data such that 
	\begin{equation}\label{b1a}
		v(t,x)\leq \varepsilon \Psi(t,x) +C_\varepsilon\;\;\text{in}\;\;\;\;[0,T_{\mathrm{max}})\times\Omega
	\end{equation}
	and
	\begin{equation}\label{b1b}
	w(t,x)+	(\tau-\varepsilon) \Psi(t,x)\leq \tau\psi(t,x)+C_\varepsilon \;\;\;\;\text{in}\;\;[0,T_{\mathrm{max}})\times\Omega.
	\end{equation}
\end{lemma}

Before proceeding with the proof of Lemma~\ref{lem3a}, let us gather in the next lemma some useful properties of the indefinite integral 
\begin{equation*}
		\Gamma(s) \triangleq \int_{v_*}^s \gamma(\sigma)\ \rd\sigma, \qquad s\geq v_*.
\end{equation*} 
of $\gamma$ when $\gamma$ is monotone non-decreasing and unbounded.

\begin{lemma}\label{lem.G}
Consider $\gamma$ satisfying~\eqref{hypgam1} and $\gamma'\geq0$ on $(0,\infty)$. Then
\begin{equation}
0 \leq \Gamma(s)\leq s\gamma(s), \qquad s\ge v_*. \label{Gam1}    
\end{equation}
Assume further that $\gamma$ satisfies~\eqref{ginfity0}. Then, for any $\varepsilon>0$, there is $s_\varepsilon>v_*$ depending only on $\gamma$ such that
\begin{equation}
    s \le \varepsilon \Gamma(s) + s_\varepsilon, \qquad s\ge v_*. \label{Gam2}    
\end{equation}
\end{lemma}

\begin{proof}
First, since $\gamma$ is non-decreasing and positive, there holds
\begin{equation*}
	0\leq \Gamma(s)\leq \gamma(s)(s-v_*)\leq s\gamma(s), \qquad s\ge v_*.
\end{equation*}   
Next, let $\varepsilon>0$. Owing to~\eqref{ginfity0}, there is $s_\varepsilon> v_*$ such that $\gamma(s)\ge 1/\varepsilon$ for all $s\ge s_\varepsilon$. Then, either $s\in [v_*,s_\varepsilon]$ and~\eqref{Gam2} is obviously true due to the non-negativity of $\Gamma$. Or $s>s_\varepsilon$ and the monotonicity and positivity of $\gamma$ imply that
\begin{equation*}
	\Gamma(s)=\int_{v_*}^s \gamma(\sigma)\ \rd\sigma=\int_{s_\varepsilon}^s \gamma(\sigma)\ \rd\sigma+\int_{v_*}^{s_\varepsilon} \gamma(\sigma)\ \rd\sigma\geq (s-s_\varepsilon)/\varepsilon,
\end{equation*}
from which~\eqref{Gam2} readily follows.
\end{proof}

\begin{proof}[Proof of Lemma~\ref{lem3a}]
It follows from~\eqref{ks2}, \eqref{Gam1} and $\gamma'\geq0$ that, for $(t,x)\in (0,T_{\mathrm{max}})\times\Omega$,
\begin{equation*}
	u\gamma(v)=\gamma(v)(\tau \partial_t v-\Delta v+v)=\mathcal{L}[\Gamma(v)]+\gamma'(v)|\nabla v|^2+v\gamma(v)-\Gamma(v)\geq \mathcal{L}[\Gamma(v)].
\end{equation*}
Thus, 
\begin{equation*}
	\mL[\Psi+e^{-t\mA/\tau}[\Gamma(v^{in})]]\geq \mL[\Gamma(v)] \;\;\;\;\text{ in }\;\; (0,T_{\mathrm{max}})\times\Omega,
\end{equation*} 
since $\mathcal{L}[\Psi]=u\gamma(v)$, and $\mL[e^{-t\mA/\tau}[\Gamma(v^{in})]]=0$. Furthermore, 
\begin{equation*}
    \nabla (\Psi+e^{-t\mA/\tau}[\Gamma(v^{in})])\cdot\mathbf{n}=\nabla\Gamma(v)\cdot\mathbf{n}=0 \;\;\;\;\text{ on }\;\;(0,T_{\mathrm{max}})\times\partial\Omega
\end{equation*}
and $\big(\Psi+e^{-t\mA/\tau}[\Gamma(v^{in})]\big)(0)=\Gamma(v^{in})$. Consequently, we deduce from the parabolic comparison principle that 
\begin{equation}\label{inc1}
	\Gamma(v)(t,x)\leq\Psi(t,x)+e^{-t\mA/\tau}[\Gamma(v^{in})]\leq \Psi(t,x)+\|\Gamma(v^{in})\|_\infty \;\;\;\;\text{in}\;\;[0,T_{\mathrm{max}})\times\Omega.
\end{equation} 

Consider now $\varepsilon>0$. Combining~\eqref{Gam2} and~\eqref{inc1} entails that
\begin{equation*}
	v(t,x)\leq \varepsilon\Psi(t,x)+s_\varepsilon +\varepsilon\|\Gamma(v^{in})\|_\infty\;\;\;\;\text{in}\;\;[0,T_{\mathrm{max}})\times\Omega.
\end{equation*}
Finally, \eqref{b1b} follows from~\eqref{corA2} and~\eqref{b1a}, and the proof is complete.
\end{proof}

\begin{remark}
If we only assume that $\gamma$ satisfies~\eqref{hypgam1} and $\gamma'\geq 0$ on $(0,\infty)$, but not~\eqref{ginfity0}, then we still obtain the following upper bound on $v$ in terms of $\Psi$: 
\begin{equation*}
	v(t,x)\leq \big(\Psi(t,x)+\|\Gamma(v^{in})\|_\infty\big)/\gamma(v_*)+v_* \;\;\;\;\text{in}\;\;[0,T_{\mathrm{max}})\times\Omega,
\end{equation*}
as a consequence of the estimate
\begin{equation*}
	\Gamma(s)=\int_{v_*}^s\gamma(\sigma)\ \rd\sigma\geq \gamma(v_*)(s-v_*).
\end{equation*}
The above upper bound on $v$ is weaker than~\eqref{b1a} and does not allow us to perform the next step of the proof provided in Lemma~\ref{prop1} below.
\end{remark}

\subsection{Asymptotically non-degenerate motility functions}\label{sec.4.2}

We now consider the generic case where $\gamma$ is not necessarily monotone non-decreasing, i.e., $\gamma'$ may change sign. Recall that we have the following asymptotically non-degenerate assumption:
\begin{equation}
	\gamma_\infty = \liminf\limits_{s\rightarrow\infty}\gamma(s)>1/\tau. \tag{1.3}
\end{equation}

\begin{lemma}\label{lem3}
Suppose that $\gamma$ satisfies \eqref{hypgam1} and \eqref{ginf}. Then, for any $\varepsilon\in(0,1)$, there is $C_\varepsilon>0$ depending on $\varepsilon$, $\gamma$ and the initial data such that 
\begin{subequations}
\begin{equation}\label{b1a2}
	v\leq 	\frac{\tau\Psi}{\varepsilon (\tau\gamma_\infty-1)+1} + C_\varepsilon \;\;\;\;\text{in}\;\;(0,T_{\mathrm{max}})\times\Omega
\end{equation}
and
\begin{equation}\label{b1b2}
	w+\frac{\varepsilon(\tau\gamma_\infty-1)\tau}{\varepsilon(\tau\gamma_\infty-1)+1}\Psi\leq\tau\psi+C_\varepsilon \;\;\;\;\text{in}\;\;(0,T_{\mathrm{max}})\times\Omega.
\end{equation}
\end{subequations}
In addition, if $\gamma_\infty=\infty$; that is,
\begin{equation*}
	\lim\limits_{s\rightarrow\infty}\gamma(s)=\infty,
\end{equation*}
then, for any $\varepsilon>0$, there exists $C_{\varepsilon}>0$ depending on $\varepsilon$, $\tau$, $\gamma$ and the initial data such that
\begin{subequations}
\begin{equation}\label{rem1a}
    v\leq \varepsilon\Psi+C_{\varepsilon} \;\;\;\;\text{in}\;\;(0,T_{\mathrm{max}})\times\Omega,
\end{equation}	
and
\begin{equation}\label{rem1b}
    w+(\tau-\varepsilon)\Psi\leq \tau\psi+C_\varepsilon \;\;\;\;\text{in}\;\;(0,T_{\mathrm{max}})\times\Omega.
\end{equation}
\end{subequations}
\end{lemma}

\begin{proof}
Let $\varepsilon>0$ and $j\ge 1$. Owing to the assumption~\eqref{ginf}, there is $s_{j,\varepsilon}>v_*$ such that
\begin{equation*}
	\gamma(s)\geq \alpha_{j,\varepsilon} \triangleq \varepsilon \min\{j , \gamma_\infty\} + (1-\varepsilon) \frac{1}{\tau},\qquad\text{for all}\;s\geq s_{j,\varepsilon}.
\end{equation*}
Moreover, by the continuity and positivity of $\gamma$ on $(0,\infty)$, we have $\beta_{j,\varepsilon} \triangleq \min_{v_*\leq s\leq s_{j,\varepsilon}}\gamma(s)\in (0,\infty)$. Therefore, we can construct a non-decreasing positive $C^1$-function $\gamma_{j,l}$ such that 
\begin{equation*}
\gamma_{j,l}= \min\{\alpha_{j,\varepsilon}/2,\beta_{j,\varepsilon}\} \;\;\text{ on }\;\; [v_*,s_{j,\varepsilon}] \;\;\text{ and }\;\; \gamma_{j,l}=\alpha_{j,\varepsilon} \;\;\text{ on }\;\; [s_{j,\varepsilon}+1,\infty).
\end{equation*}
Notice that  $\gamma_{j,l}\leq \gamma$ on $[v_*,\infty)$. We then argue as in the proof of Lemma~\ref{lem3a} to deduce that
\begin{equation*}
	u\gamma(v)\geq u\gamma_{j,l}(v) = (\tau\partial_t v-\Delta v+v)\gamma_{j,l}(v)
    \geq \tau \partial_t\Gamma_{j,l}(v)-\Delta \Gamma_{j,l}(v)+\Gamma_{j,l}(v)
\end{equation*}
in $(0,T_{\mathrm{max}})\times\Omega$, where 
\begin{equation*}
    \Gamma_{j,l}(s):=\int_{v_*}^s \gamma_{j,l}(\sigma)\ \rd \sigma \le \Gamma(s), \qquad s\ge v_*,
\end{equation*}
and, subsequently, that
\begin{equation*}
	\Gamma_{j,l}(v)(t,x)\leq \Psi(t,x)+\|\Gamma_{j,l}(v^{in})\|_\infty \le \Psi(t,x)+\|\Gamma(v^{in})\|_\infty \;\;\;\;\text{in}\;\;(0,T_{\mathrm{max}})\times\Omega,
\end{equation*}
by the parabolic comparison principle. Now, we notice that, for $s\geq s_{j,\varepsilon}+1$, the positivity of $\gamma_{j,l}$ ensures that
\begin{equation*}
	\Gamma_{j,l}(s) \geq \int_{s_{j,\varepsilon}+1}^s \gamma_{j,l}(\sigma) \ \rd\sigma = \alpha_{j,\varepsilon}(s-s_{j,\varepsilon}-1).
\end{equation*}
It then follows that, for $(t,x)\in [0,T_{\mathrm{max}})\times\Omega$,
\begin{equation*}
	\alpha_{j,\varepsilon} (v(t,x)-s_{j,\varepsilon}-1)_+\leq \Gamma_{j,l}(v)(t,x)\leq \Psi(t,x)+\|\Gamma(v^{in})\|_\infty,
\end{equation*}
whence
\begin{equation}\label{n3a}
	v(t,x)\leq  \frac{\tau\Psi(t,x)}{\varepsilon\big( \tau \min\{j,\gamma_\infty\} - 1 \big)+ 1} +\frac{\tau\|\Gamma(v^{in})\|_\infty}{\varepsilon\big( \tau \min\{j,\gamma_\infty\} - 1 \big)+ 1} + s_{j,\varepsilon} +1.
\end{equation}
At this point, either $\gamma_\infty<\infty$ and we fix any $j\ge\gamma_\infty$ in~\eqref{n3a} to obtain~\eqref{b1a2}, from which~\eqref{b1b2} follows thanks to~\eqref{corA2}. Or $\gamma_\infty=\infty$ and we pick $j_\varepsilon$ large enough satisfying $\tau-\varepsilon\le \varepsilon^2(\tau j_\varepsilon -1)$ and apply~\eqref{n3a} with $j=j_\varepsilon$ to obtain
\begin{equation*}
    v(t,x)\leq \varepsilon \Psi(t,x) +\frac{\tau\|\Gamma(v^{in})\|_\infty}{\varepsilon\big( \tau j_\varepsilon - 1 \big)+ 1} + s_{j_\varepsilon,\varepsilon} +1 ,\;\;\;\text{in}\;\;(0,T_{\mathrm{max}})\times\Omega.  
\end{equation*}
We have shown~\eqref{rem1a} and we complete the proof by deducing~\eqref{rem1b} from~\eqref{corA2} as above.
\end{proof}

\begin{lemma}\label{prop1}
Under the assumptions of, either Lemma~\ref{lem3a}, or Lemma~\ref{lem3}, there is a positive constant $C>0$ depending on $\tau$, $\gamma$ and the initial data such that 
\begin{equation*}
	\|\Psi(t)\|_\infty\leq C\|\Psi(t)\|_1+C, \qquad t\in [0,T_{\mathrm{max}}).
\end{equation*}
\end{lemma}

\begin{proof} 
Due to the non-negativity of $w$, one deduces from either~\eqref{b1b} (with $\varepsilon=\tau/2$) in Lemma~\ref{lem3a}, or~\eqref{b1b2} (with $\varepsilon=1/2$) in Lemma~\ref{lem3} that there is $C>0$ depending on $\tau$, $\gamma$ and the initial data such that 
\begin{equation*}
	 \Psi(t,x)\leq C\psi(t,x)+C=C\mathcal{A}^{-1}[\Psi](t,x)+C, \qquad  (t,x) \in  [0,T_{\text{max}}) \times\Omega.
\end{equation*}
Consequently, by the elliptic comparison principle, for each   $k\in\mathbb{N}\cup\{0\}$,
\begin{equation}\label{suc}
	\mA^{-k}[\Psi](t,x)\leq C\mA^{-k-1}[\Psi](t,x)+C,\qquad (t,x) \in  [0,T_{\text{max}}) \times\Omega.
\end{equation}
Iterating~\eqref{suc}, we are led to
\begin{equation}
   \Psi(t,x) \le C^k \mA^{-k}[\Psi](t,x) + \frac{C^{k+1}-1}{C-1},\qquad (t,x) \in  [0,T_{\text{max}}) \times\Omega. \label{X10}
\end{equation}

Next, we claim that there is a finite integer $k_0>0$ such that $\mA^{-k_0}[\Psi]\leq C\|\Psi\|_1$. Indeed, we infer from Lemma~\ref{lm2} that $\|\mA^{-1}[\Psi]\|_{p_1}\leq C(p_1)\|\Psi\|_1$ for any $p_1\in(1,N/(N-2)_+)$. Then classical regularity theory of elliptic equations indicates that $\|\mA^{-k}[\Psi]\|_{W^{2(k-1),p_1}}\leq C(k,p_1)\|\mA^{-1}[\Psi]\|_{p_1}$ for any integer $k\ge 1$. Now, choosing $k=k_0\triangleq (N+2)/2$ and noting that $W^{2(k_0-1),p_1}(\Omega)$ embeds continuously in $L^\infty(\Omega)$, we end up with 
\begin{equation*}
   \|\mA^{-k_0}[\Psi]\|_\infty\leq C(k_0)\|\Psi\|_1. 
\end{equation*}
Combining the above inequality with~\eqref{X10} (with $k=k_0$) completes the proof.
\end{proof}

According to Lemma~\ref{lem3a}, Lemma~\ref{lem3} and  Lemma~\ref{prop1}, we finally arrive at the main result of this section.

\begin{proposition} \label{cor1}Under the assumptions of  Lemma~\ref{prop1}, there is a positive constant $C_0>0$ depending on $\tau$, $\gamma$ and the initial data such that 
\begin{equation*}
	\|v(t)\|_\infty\leq C_0\|\Psi(t)\|_1+C_0, \qquad t\in [0,T_{\mathrm{max}}).
\end{equation*}
In addition, if $\gamma_\infty=\infty$, then for any $\varepsilon>0$, there is $C_\varepsilon>0$ depending on $\varepsilon$, $\tau$, $\gamma$ and the initial data such that 
\begin{equation*}
	\|v(t)\|_\infty\leq \varepsilon\|\Psi(t)\|_1+C_\varepsilon, \qquad t\in [0,T_{\mathrm{max}}).
\end{equation*}
\end{proposition}

\begin{proof}
The first statement readily follows from~\eqref{b1a2} and Lemma~\ref{prop1}, while the second one is an immediate consequence of~\eqref{b1a}, \eqref{rem1a} and Lemma~\ref{prop1}.
\end{proof}

\section{Uniform-in-time $L^1$-bounds for $\Psi$}\label{sec.5}

According to Proposition~\ref{cor1}, in order to prove the uniform boundedness of $v$, it suffices to derive a uniform $L^1$-bound for $\Psi$, which is achieved by the following delicate energy estimates.

\subsection{Monotone non-decreasing motility functions} \label{sec.5.1}

First, we consider the case $\gamma'\geq0$, a feature which allows us to construct a Lyapunov functional. Before stating our result, let us recall that, given $z\in L^2(\Omega)$ with
\begin{equation*}
    \langle z \rangle \triangleq \frac{1}{|\Omega|} \int_\Omega z(x)\rd x = 0,
\end{equation*}
there is a unique solution $\mathcal{K}[z]\in H^2(\Omega)$ to
\begin{subequations}\label{mK}
\begin{equation}
    - \Delta \mK[z] = z \;\;\;\;\text{ in }\;\;\Omega, \quad \nabla\mK[z]\cdot \mathbf{n} = 0\;\;\;\;\text{ on }\;\;\partial\Omega, \label{mK1}
\end{equation}
satisfying
\begin{equation}
    \langle \mK[z] \rangle = 0. \label{mK2}
\end{equation}
\end{subequations}

\begin{lemma}\label{lem.lf}
Recall that $m=\|u^{in}\|_1/|\Omega| = \langle u^{in}\rangle$, see~\eqref{e0}, and set
\begin{equation*}
	\mathcal{I}(u,v) \triangleq \frac{1}{2} \|\nabla \mathcal{K}[u-m]\|_2^2 + m \tau  \|\Gamma(v)\|_1 - m \gamma(m) \tau \|v\|_1. 
\end{equation*}
Then there holds
	\begin{equation}
		\frac{\rd}{\rd t} \mathcal{I}(u(t),v(t)) + \mathcal{D}(u(t),v(t)) = 0, \qquad t\in [0,T_{\mathrm{max}}), \label{D1} 
	\end{equation}
	with
	\begin{align*}
		\mathcal{D}(u,v) & \triangleq m \int_\Omega \gamma'(v) |\nabla v|^2\ \rd x + \int_\Omega \gamma(v) (u-m)^2\ \rd x \\
		& \qquad + m \int_\Omega (v-m)(\gamma(v)-\gamma(m))\ \rd x \ge 0.
	\end{align*}
\end{lemma}

\begin{proof}First, the non-negativity of $\mathcal{D}$ follows from the monotonicity of $\gamma$. By~\eqref{ks1} and~\eqref{mK}, a direct computation yields that
\begin{align*}
	\frac12\frac{\rd}{\rd t}\|\nabla\mathcal{K}[u-m]\|_2^2=&\int_\Omega u\gamma(v)(m-u)\ \rd x\\
	=&-\int_\Omega \gamma(v)(u-m)^2 \ \rd x-m\int_\Omega \gamma(v)(u-m)\ \rd x,
\end{align*} 
while, using~\eqref{ks2},
\begin{align*}
	m\tau\frac{\rd}{\rd t}\int_\Omega \Gamma(v)\ \rd x&=m\int_\Omega \tau\gamma(v)\partial_t v\ \rd x\\
	&=m\int_\Omega\gamma(v)(u-v+\Delta v)\ \rd x\\
	&=m\int_\Omega \gamma(v)(u-m)\ \rd x-m\int_\Omega\gamma'(v)|\nabla v|^2\ \rd x-m\int_\Omega \gamma(v)(v-m)\ \rd x\\
	&=m\int_\Omega \gamma(v)(u-m)\ \rd x-m\int_\Omega\gamma'(v)|\nabla v|^2\ \rd x-m\int_\Omega \big(\gamma(v)-\gamma(m)\big)(v-m)\ \rd x\\
	&\quad - m\gamma(m)\int_\Omega(v-m)\ \rd x.
\end{align*}
Also, by integration of~\eqref{ks2} over $\Omega$, together with \eqref{e0},
\begin{equation*}
	\tau\frac{\rd}{\rd t}\int_\Omega v \ \rd x= -\int_\Omega (v-u) \ \rd x=- \int_\Omega (v-m) \ \rd x.
\end{equation*}
Collecting the above identities gives~\eqref{D1} and completes the proof.
\end{proof}

Thanks to the positivity of $\gamma$ and the non-negativity of $\gamma'$, the three terms involved in $\mathcal{D}(u,v)$ are non-negative, while the control~\eqref{nup} on $\|v\|_1$ guarantees that $\mathcal{I}(u,v)$ is bounded from below. Hence, there is a positive constant $C_1$ depending only on $\gamma$ and the initial data such that
\begin{equation}
	\begin{split}
		- C_1 \le \mathcal{I}(u(t),v(t) ) & \le \mathcal{I}(u^{in},v^{in}), \qquad t\in [0,T_{\mathrm{max}}), \\
		\int_0^{T_{\mathrm{max}}} \mathcal{D}(u(s),v(s))\ \mathrm{d}s & \le \mathcal{I}(u^{in},v^{in}) + C_1.
	\end{split} \label{P3}
\end{equation}

\begin{lemma}\label{lem1}
	There holds
	\begin{equation*}
		m\|\gamma(v(t))\|_1\leq\int_\Omega(\gamma(v(t))-\gamma(m))(v(t)-m)\rd x+m|\Omega|\gamma(2m), \qquad t\in [0,T_{\mathrm{max}}).
	\end{equation*}
\end{lemma}

\begin{proof}
	Owing to the monotonicity of $\gamma$, 
	\begin{align*}
		& \int_\Omega (\gamma(v)-\gamma(m))(v-m)\ \mathrm{d}x \nonumber\\
		& \qquad \ge \int_\Omega \mathbf{1}_{[2m,\infty)}(v) (\gamma(v)-\gamma(m))(v-m)\ \mathrm{d}x \nonumber\\
		& \qquad\ge m \int_\Omega \mathbf{1}_{[2m,\infty)}(v) (\gamma(v)-\gamma(m))\ \mathrm{d}x \nonumber\\
		& \qquad = m \|\gamma(v)\|_1 - m \int_\Omega \mathbf{1}_{(0,2m)}(v) \gamma(v)\ \mathrm{d}x - m \gamma(m) \int_\Omega \mathbf{1}_{[2m,\infty)}(v)\ \mathrm{d}x \nonumber\\
		& \qquad \ge m \|\gamma(v)\|_1 - m \gamma(2m) \int_\Omega \mathbf{1}_{(0,2m)}(v)\ \mathrm{d}x - m \gamma(2m) \int_\Omega \mathbf{1}_{[2m,\infty)}(v)\ \mathrm{d}x \nonumber\\
		& \qquad \ge m \|\gamma(v)\|_1 - m \gamma(2m) |\Omega|. 
	\end{align*}
	This completes the proof.
\end{proof}

Now we are ready to derive a uniform bound for $\|\Psi\|_1$.

\begin{lemma}\label{L1a}Assume \eqref{hypgam1} and that  $\gamma'\geq0$. There is $C>0$ depending on $\gamma$, $\tau$ and the initial data such that
	\begin{equation*}
		\sup_{0\leq t<T_{\mathrm{max}}}\|\Psi(t)\|_1\leq C.
	\end{equation*}
\end{lemma}

\begin{proof}First, an integration of \eqref{Psi} yields that
	\begin{equation}
		\tau \frac{\mathrm{d}}{\mathrm{d}t} \|\Psi\|_1 + \|\Psi\|_1 = \int_\Omega u\gamma(v)\ \mathrm{d}x. \label{P1}
	\end{equation} 
It then follows from~\eqref{P1}, Lemma \ref{lem1}, and Young's inequality that
\begin{align*}
	\tau \frac{\mathrm{d}}{\mathrm{d}t} \|\Psi\|_1 + \|\Psi\|_1 & = \int_\Omega (u-m) \gamma(v)\ \mathrm{d}x + m \|\gamma(v)\|_1 \\
	& \le \frac{1}{4} \int_\Omega (u-m)^2 \gamma(v)\ \mathrm{d}x + (m+1) \|\gamma(v)\|_1 \\
	& \le \int_\Omega (u-m)^2 \gamma(v)\ \mathrm{d}x + \frac{m+1}{m} \int_\Omega (\gamma(v)-\gamma(m))(v-m)\ \mathrm{d}x \\
	& \qquad + (m+1) \gamma(2m) |\Omega| \\
	& \le \left( 1 + \frac{m+1}{m^2} \right) \mathcal{D}(u,v) + (m+1) \gamma(2m) |\Omega|.
\end{align*}
Hence,
\begin{equation*}
	\frac{\mathrm{d}}{\mathrm{d}t} \left( e^{t/\tau} \|\Psi(t)\|_1 \right) \le \frac{C_2}{\tau} e^{t/\tau} (1+\mathcal{D}(u(t),v(t))), \qquad t\in [0,T_{\mathrm{max}}).
\end{equation*}
Integrating the above differential inequality with respect to time and using~\eqref{P3} give, for $t\in [0,T_{\mathrm{max}})$
\begin{align*}
	\|\Psi(t)\|_1 & \le C_2 (1-e^{-t/\tau}) + \frac{C_2}{\tau} \int_0^t e^{(s-t)/\tau} \mathcal{D}(u(s),v(s))\ \mathrm{d}s \\
	& \le  C_2 + \frac{C_2}{\tau} \int_0^{T_{\mathrm{max}}} \mathcal{D}(u(s),v(s))\ \mathrm{d}s \\
	& \le   C_2 + \frac{C_2}{\tau} (C_1 + \mathcal{I}(u^{in},v^{in})).
\end{align*}
This completes the proof.
\end{proof}

\subsection{Asymptotically non-degenerate motility functions}\label{sec.5.2}

We now turn to generic motility functions $\gamma$ satisfying~\eqref{hypgam1} and~\eqref{ginf}, but not necessarily monotone. In this case, the key step is to establish an estimate for a functional involving a nonlinear coupling term of the unknowns and the auxiliary functions.

\begin{lemma}\label{lem.PsiL1}
There is $K>0$ depending only on $\tau$, $\gamma$ and the initial data such that
\begin{equation*}
	\|\Psi(t)\|_1 \le K, \qquad t\in [0,T_{\mathrm{max}}).
\end{equation*}
\end{lemma}

\begin{proof}
Since $\tau\gamma_\infty>1$ by~\eqref{ginf} and $\Psi$ is non-negative, it readily follows from~\eqref{b1b2} (with $\varepsilon=1/2$) that there is $K_0>0$ such that
	\begin{equation}
		\tau \psi - w \ge - K_0 \;\;\text{ in }\;\; (0,T_{\mathrm{max}})\times\Omega. \label{X1}
	\end{equation}
	We next recall that the following identities
	\begin{equation}
		\partial_t w + u \gamma(v) = \mathcal{A}^{-1}[u\gamma(v)] = \mathcal{L}[\psi] = \tau \partial_t \psi + \Psi \;\;\text{ in }\;\; (0,T_{\mathrm{max}})\times\Omega, \label{X2}
	\end{equation}
	which follow from~\eqref{Psi}, \eqref{keyid}, and the definition of $\psi$. On the one hand, we infer from~\eqref{ks1}, \eqref{kid2}, and~\eqref{X2} that
	\begin{align}
		\frac{\rd}{\rd t} \int_\Omega u (\tau\psi - w + K_0)\ \rd x & = \int_\Omega (\tau\psi - w + K_0) \Delta (u\gamma(v))\ \rd x + \int_\Omega u (\tau\partial_t \psi - \partial_t w)\ \rd x \nonumber\\
		& = \int_\Omega u\gamma(v) \Delta(\tau\psi-w)\ \rd x + \int_\Omega u(u\gamma(v)-\Psi)\ \rd x \nonumber\\
		& = - \int_\Omega u\gamma(v) \mathcal{A}[\tau\psi-w]\ \rd x + \int_\Omega u\gamma(v) (\tau\psi-w)\ \rd x \nonumber\\
		& \qquad + \int_\Omega u(u\gamma(v)-\Psi)\ \rd x \nonumber\\
		& = - \int_\Omega u\gamma(v) (\tau\Psi - u + w - \tau\psi - u)\ \rd x - \int_\Omega u \Psi\ \rd x \nonumber\\
		& = 2 \int_\Omega u^2 \gamma(v)\ \rd x - \int_\Omega u\gamma(v) (v+\eta)\ \rd x - \int_\Omega u\Psi\ dx ,\nonumber 
	\end{align}
    where we have also used the identity that  $\mathcal{A}[\tau\psi-w]=\tau\Psi-u$ following from the definitions of the auxiliary functions.	On the other hand, it follows from~\eqref{ks1}, \eqref{Psi}, and the definition~\eqref{mK} of $\mathcal{K}$ that
	\begin{equation*}
		\frac{\rd}{\rd t} \|\nabla\mathcal{K}[u-m]\|_2^2 + 2 \int_\Omega u^2 \gamma(v)\ \rd x = 2m \int_\Omega u\gamma(v)\ \rd x
	\end{equation*}
	and
	\begin{equation*}
		\tau \frac{\rd}{\rd t} \|\Psi\|_1 + \|\Psi\|_1 = \int_\Omega u\gamma(v)\ \rd x.
	\end{equation*}
	Combining the above differential equations and introducing
	\begin{equation*}
		\mathcal{F} := \|\nabla\mathcal{K}[u-m]\|_2^2 + \frac{1}{2} \int_\Omega u (\tau\psi - w + K_0)\ \rd x + \tau\|\Psi\|_1,
	\end{equation*}
	we note that $\mathcal{F}\ge 0$ by~\eqref{X1} and obtain
	\begin{equation}
		\begin{split}
			& \frac{\rd \mathcal{F}}{\rd t} + \int_\Omega u^2 \gamma(v)\ \rd x + \frac{1}{2} \int_\Omega uv\gamma(v)\ \rd x + \frac{1}{2} \int_\Omega u\Psi\ \rd x + \|\Psi\|_1 \\
			& \qquad = (2m+1) \int_\Omega u\gamma(v)\ \rd x - \frac{1}{2} \int_\Omega \eta u \gamma(v)\ \rd x.
		\end{split} \label{X4}
	\end{equation}
	Owing to~\eqref{etabd}, 
	\begin{align*}
		(2m+1) \int_\Omega u\gamma(v)\ \rd x - \frac{1}{2} \int_\Omega \eta u \gamma(v)\ \rd x & \le (2m+1+\|\eta\|_\infty) \int_\Omega u\gamma(v)\ \rd x \\
		& \le K_1  \int_\Omega u\gamma(v)\ \rd x,
	\end{align*}
	where $K_1 := 2m+1+ \max\{\|u^{in}\|_\infty , \|v^{in}\|_\infty\}$. Splitting the integral on the right hand side of the above inequality gives
	\begin{align*}
		(2m+1) \int_\Omega u\gamma(v)\ \rd x - \frac{1}{2} \int_\Omega \eta u \gamma(v)\ \rd x & \le K_1 \int_\Omega \left[ \mathbf{1}_{[v_*,2K_1]}(v) u \gamma(v) + \mathbf{1}_{(2K_1,\infty)}(v) u \gamma(v) \right]\ \rd x \\
		& \le K_1 \sup_{[v_*,2K_1]}\{\gamma\} \int_\Omega u\ \rd x + \frac{1}{2} \int_\Omega uv\gamma(v)\ \rd x \\
		& \le \frac{1}{2} \int_\Omega uv\gamma(v)\ \rd x + K_2,
	\end{align*}
	with $K_2 := m K_1 |\Omega| \sup_{[v_*,2K_1]}\{\gamma\}$. Thus, \eqref{X4} becomes
	\begin{equation}
		\frac{\rd\mathcal{F}}{\rd t} + \int_\Omega u^2 \gamma(v)\ \rd x + \frac{1}{2} \int_\Omega u\Psi\ \rd x + \|\Psi\|_1 \le K_2. \label{X5}
	\end{equation}
	We now note that the positivity of $\gamma_\infty$ and $\gamma$, along with the continuity of $\gamma$, entails that
	\begin{equation*}
		\gamma_* := \inf_{[v_*,\infty)}\{\gamma\}>0,
	\end{equation*}
	so that, by the Poincar\'e-Wirtinger inequality,
	\begin{align}
		\int_\Omega u^2 \gamma(v)\ \rd x & \ge \gamma_* \int_\Omega (u-m+m)^2\ \rd x \ge \gamma_* \int_\Omega \left[ \frac{(u-m)^2}{2} - m^2 \right]\ \rd x \nonumber\\
		& = \frac{\gamma_*}{2} \|u-m\|_2^2 - \gamma_* m^2 |\Omega| \ge \gamma_* K_3 \|\nabla\mathcal{K}[u-m]\|_2^2 - \gamma_* m^2 |\Omega| \label{X6}
	\end{align}
	for some $K_3>0$ depending only on $\Omega$. Also, by~\eqref{etabd} and~\eqref{kid2},
	\begin{align}
		\tau \int_\Omega u\Psi \ \rd x & = \int_\Omega u (v+\tau\psi - w +\eta)\ \rd x \nonumber\\
		& \ge \int_\Omega u(\tau\psi-w+K_0)\ \rd x - (K_0 + \|\eta\|_\infty) \int_\Omega u\ \rd x \nonumber\\
		& \ge \int_\Omega u(\tau\psi-w+K_0)\ \rd x - K_4, \label{X7}
	\end{align}
	with $K_4 := m (K_0+\max\{\|u^{in}\|_\infty , \|v^{in}\|_\infty \}) |\Omega|$.

Collecting~\eqref{X6} and~\eqref{X7}, we are led to the lower bound
\begin{equation*}
	\int_\Omega u^2 \gamma(v)\ \rd x + \frac{1}{2} \int_\Omega u\Psi\ \rd x + \|\Psi\|_1 \ge K_5 \mathcal{F} - K_6,
\end{equation*}
for some positive constants $K_5<1<K_6$ depending only on $\gamma$, $\tau$ and the initial data. Combining the just derived lower bound with~\eqref{X5}, we find
\begin{equation*}
	\frac{\rd\mathcal{F}}{\rd t} + K_5 \mathcal{F} \le K_7 := K_2 + K_6,
\end{equation*}
whence
\begin{equation*}
	\tau\|\Psi(t)\|_1 \le \mathcal{F}(t) \le \max\left\{\mathcal{F}(0), \frac{K_7}{K_5} \right\}, \qquad t\in [0,T_{\mathrm{max}}),
\end{equation*}
the lower bound being a consequence of~\eqref{X1}. This completes the proof.
\end{proof}

\section{Proofs of Theorem~\ref{TH1} and Theorem~\ref{TH2}}\label{sec.6}

\begin{proof}[Proof of Theorem~\ref{TH1}]
Under the assumption of Theorem~\ref{TH1}, we have already shown the uniform-in-time  boundedness of $v$ on $[0,T_{\mathrm{max}})$, according to Lemma~\ref{lem.PsiL1} and Proposition~\ref{cor1}, and the bound does not depend on $T_{\mathrm{max}}$. Then the uniform-in-time boundedness of $u$ on $[0,T_{\mathrm{max}})$ follows from Proposition~\ref{prop2.2} and, since the bound does not depend on $T_{\mathrm{max}}$, it also excludes the finiteness of $T_{\mathrm{max}}$, according to Proposition~\ref{local}. This completes the proof.
\end{proof}

\begin{proof}[Proof of Theorem~\ref{TH2}]
Assume that $\gamma'\geq0$. We first point out that, if $\gamma\in L^\infty(v_*,\infty)$, the positive lower bound on $\gamma(v)$ stemming from the positivity~\eqref{hypgam1} of $\gamma$ and~\eqref{e00} allows us to apply \cite[Theorem~1.2]{XiaoJiang2022} to conclude the proof.

Otherwise, $\lim\limits_{s\rightarrow\infty}\gamma(s)=\infty$. The proof of the global well-posedness and uniform-in-time boundedness statements in Theorem~\ref{TH2} is then the same as that of Theorem~\ref{TH1}, thanks to Lemma~\ref{L1a} and Proposition~\ref{cor1}.

The convergence to the homogeneous steady state can be verified in the same manner as done in \cite[Section~7.2]{JLZ2022}, to which we refer. This completes the proof.
\end{proof}

\section*{Acknowledgments}
\noindent JJ is supported by National Natural Science Foundation of China (NSFC)
under grants No.~12271505 \& No.~12071084. The work of PhL is partially funded by the Chinese Academy of Sciences President's International Fellowship Initiative Grant No.~2025PVA0101.

\bibliographystyle{siam}
\bibliography{KS2023}
\end{document}